\documentclass[12pt,psfig, reqno]{amsart}
\usepackage{amsmath, amsthm, amsfonts, amssymb, latexsym, graphicx}
\usepackage[latin1]{inputenc}

\newtheorem{theorem}{Theorem}[section]
\newtheorem{lemma}[theorem]{Lemma}

\newtheorem{corollary}[theorem]{Corollary}
\theoremstyle{definition}
\newtheorem{definition}[theorem]{Definition}
\newtheorem{example}[theorem]{Example}

\newtheorem{question}[theorem]{Question}

\newtheorem{remark}[theorem]{Remark}
\numberwithin{equation}{section}

\newcommand{\mC}{\ensuremath{\mathbb{C}}}
\newcommand{\mD}{\ensuremath{\mathbb{D}}}

\newcommand{\mN}{\ensuremath{\mathbb{N}}}

 \allowdisplaybreaks
\begin{document}

\title{Fatou and Julia like sets II}
\author[K. S. Charak]{Kuldeep Singh Charak}
\address{
\begin{tabular}{lll}
& Kuldeep Singh Charak\\
& Department of Mathematics\\
& University of Jammu\\
& Jammu-180 006\\
& India
\end{tabular}}
\email{kscharak7@rediffmail.com }

\author[A. Singh]{Anil Singh}
\address{
\begin{tabular}{lll}
& Anil Singh\\
& Department of Mathematics\\
& University of Jammu\\
& Jammu-180 006\\
& India
\end{tabular}}
\email{anilmanhasfeb90@gmail.com }

\author[M. Kumar]{Manish Kumar}
\address{
\begin{tabular}{lll}
& Manish Kumar\\
& Department of Mathematics\\
& University of Jammu\\
& Jammu-180 006\\ 
& India\\
\end{tabular}}
\email{manishbarmaan@gmail.com}

\begin{abstract}  This paper is a continuation of authors work: {\it Fatou and Julia like sets, Ukranian J. Math., to appear/arXiv:2006.08308[math.CV](see \cite{anil})}. Here, we introduce escaping like set and generalized escaping like set for a family of holomorphic functions on an arbitrary domain, and establish some distinctive properties of these sets. The connectedness of the Julia like set is also proved.
\end{abstract}

\renewcommand{\thefootnote}{\fnsymbol{footnote}}
\footnotetext{The work of the first author is partially supported by Mathematical Research Impact Centric Support (MATRICS) grant, File No. MTR/2018/000446, by the Science and Engineering Research Board (SERB), Department of Science and Technology (DST), Government of India.}

\maketitle

\section{introduction}
Let $f$ be an entire function. Then the \textit{Fatou set} of $f$, denoted by $F(f)$ is a subset of $\mC$ in which the family $\{f^n:n\geq 1\}$ of iterates of $f$ is normal, and the complement $\mathbb{C}\setminus F(f),$  denoted by $J(f),$ is called the \textit{Julia set } of $f.$  $F(f)$ is an open subset of $\mathbb{C}$ and $J(f)$ is a closed subset of $\mathbb{C}$, and both are completely invariant sets under $f$. The study of Fatou and Julia sets of holomorphic functions is a subject matter of Complex Dynamics for which one can refer to \cite{Bergweiler-5, gamelin,  steinmetz}. The Fatou and Julia theory is extended to semigroups of rational functions (\cite{Hinkkanen1, Hinkkanen2}) and transcendental entire functions (see \cite{poon1, poon2}).

\medskip

Throughout, we shall denote by $\mathcal{H}\left(D\right)$, the class of all holomorphic functions on a domain $D\subseteq \mC$ and $\mD$ shall denote the open unit disk in $\mC.$

\smallskip

For an arbitrary domain $D\subseteq \mC$ and a subfamily $\mathcal{F}$ of $\mathcal{H}\left(D\right)$, the authors in \cite{anil} introduced Fatou like set $F(\mathcal{F})$ and Julia like set $J(\mathcal{F})$ of the family $\mathcal{F}$ as follows:
Fatou like set $F(\mathcal{F})$ of $\mathcal{F}$ is defined to be a subset of $D$ on which $\mathcal{F}$ is normal and Julia like set $J(\mathcal{F})$ of $\mathcal{F}$ is the complement $D\setminus F(\mathcal{F})$ of $F(\mathcal{F})$. If $\mathcal{F}$ happens to be a family of iterates of an entire function $f$, then $F(\mathcal{F})$ and $J(\mathcal{F})$ reduce to the Fatou set of $f$ and the Julia set of $f$ respectively. Various interesting properties of the sets $F(\mathcal{F})$ and $J(\mathcal{F})$ are studied in \cite{anil}.

\medskip

In this paper we extend the work done in \cite{anil} and introduce the escaping like set and generalized escaping like set for a family of holomorphic functions on an arbitrary domain. We have divided our findings into three sections: In Section $2$, we present some interesting properties of Julia like set including its connectedness, in Section $3$, we introduce escaping like set, generalized escaping like set and prove some distinctive properties of these sets, and finally in Section $4$, we have some discussion on limit functions and fixed points of $\mathcal{F}.$

\section{ Properties of Julia like set $J(\mathcal{F})$ of $\mathcal{F}$}\label{sec2}
Let $\mathcal{F}$ be a subfamily of $\mathcal{H}(D)$ and $z\in\mathbb{C}.$ We define the {\it backward orbit} of $z$ with respect to $\mathcal{F}$ as 
  $$\mathcal{O}^{-}_{\mathcal{F}}(z):=\left\{w\in D:f(w)=z,\mbox{ for some }f\in\mathcal{F}\right\}$$ 
	and the {\it exceptional set } of $\mathcal{F}$ is defined as
	 $$E(\mathcal{F}):=\left\{z\in\mathbb{C}:\mathcal{O}^{-}_{\mathcal{F}}(z)\mbox{ is finite }\right\}.$$
 
If $\mathcal{F}$ is a semigroup of entire functions and $z\in J(\mathcal{F})\setminus E(\mathcal{F})$, then the backward invariance of $J(\mathcal{F})$ (see \cite{Hinkkanen1}, Theorem 2.1) implies that $\mathcal{O}^{-}_{\mathcal{F}}(z)\subseteq J(\mathcal{F})$. The other way inclusion is true for any $\mathcal{F}\subseteq \mathcal{H}(D)$ with $E(\mathcal{F})\neq \emptyset$ (see, \cite{anil}, Theorem 1.9). Thus we have:
\begin{theorem} Suppose that $\mathcal{F}$ is a semigroup of entire functions with $E(\mathcal{F})\neq \emptyset$,  and $z\in J(\mathcal{F})\setminus E(\mathcal{F})$. Then $J(\mathcal{F})={\overline{\mathcal{O}^{-}_{\mathcal{F}}(z)}}.$
\end{theorem}

We know (see \cite{anil}, Theorem 1.1) that if $N$ is a neighborhood of a point  $z_0\in J(\mathcal{F})$, then $\mathbb{C}\setminus \cup_{f\in\mathcal{F}}f\left(N\right)$ contains at  most one point. If $J(\mathcal{F})$ has an isolated point, the following counterpart holds: 

\begin{theorem}\label{expansive} 
\begin{itemize}
\item[(a)] Suppose that $J(\mathcal{F})$ has an isolated point. Then $\mathbb{C}\setminus \bigcup_{f\in\mathcal{F}}f\left(U\right)$ has at most one point, for some open set $U\subseteq F(\mathcal{F})$.\\
\item[(b)] Suppose that $N$ is a neighborhood of a point in $J(\mathcal{F})$. If $E(\mathcal{F})\neq\phi$, then 
$$\mathbb{C}\setminus \left(\bigcup_{f\in\mathcal{F}}f(N)\right)\subset E(\mathcal{F}).$$
\end{itemize}
\end{theorem}

\begin{proof} 
 Let $z_0\in J(\mathcal{F})$ be an isolated point. Then we can choose a neighborhood $N$ of $z_0$ such that $U:=N\setminus\left\{z_0\right\}\subseteq F(\mathcal{F}).$ Since $\mathcal{F}$ is not normal at $z_0$, by an extension of Montel's theorem (see \cite{cara}, p. $203$), $\mathcal{F}$ omits at most one point in $N\setminus\left\{z_0\right\}$. This prove $(a).$\\
 Let $N$ be a neighborhood of $z_0\in J(\mathcal{F})$ and $w_0\in \mathbb{C}\setminus \bigcup_{f\in\mathcal{F}}f(N)$. Suppose that  $w_0\notin E(\mathcal{F})$ and let $w_1\in E(\mathcal{F})$. Then we can choose a deleted neighborhood $N_1\subset N$ of $z_0$ such that $\mathcal{O}^{-}_{\mathcal{F}}(w_1)\cap N_1=\phi$ showing that $\mathcal{F}$ omits two points $w_0, w_1$ in the deleted neighborhood $N_1$ of $z_0$. Now by extension of Montel's Theorem, $z_0\in F(\mathcal{F})$, a contradiction. This proves $(b).$ 
\end{proof}

\begin{example} Consider the family $\mathcal{F}:=\left\{nz:n\in\mathbb{N}\right\}$ of entire functions . Then $F(\mathcal{F})=\mathbb{C}\setminus\left\{0\right\}.$ For any deleted neighborhood $N$ of $0$, $\bigcup_{f\in\mathcal{F}}f(N)=\mathbb{C}\setminus\left\{0\right\}$ and the set $\mathbb{C}\setminus\left(\cup_{f\in\mathcal{F}}f(N)\right)$ contains exactly one point.
\end{example}

\begin{theorem}\label{perfect} Let $\mathcal{F}$ be a family of transcendental entire functions with nonempty backward invariant Julia like set $J(\mathcal{F}).$ Then $J(\mathcal{F})$ is a singleton or an infinite set. If $J(\mathcal{F})$ is a singleton $\left\{z_0\right\},$ say, then for any $f\in\mathcal{F}$, $z_0$ is a fixed point of $f$ or a Picard exceptional value of $f$,  and if $J(\mathcal{F})$ is infinite, then $J(\mathcal{F})$ has no isolated points.
\end{theorem}
\begin{proof} Suppose that $J(\mathcal{F})$ is finite and has at least two points. Then there is some $z\in J(\mathcal{F})$ and $f\in\mathcal{F}$ such that $f^{-1}(\{z\})$ is infinite. Backward invariance of $J(\mathcal{F})$ implies that $f^{-1}(\{z\})\subseteq J(\mathcal{F})$, which is a contradiction. Hence $J(\mathcal{F})$ reduces to a singleton, $\left\{z_0\right\},$ say. Since for any $f\in\mathcal{F}$ $f^{-1}(\{z_0\})\subseteq J(\mathcal{F})$, $f^{-1}(\{z_0\})=\left\{z_0\right\}$ or $f^{-1}(\{z_0\})=\emptyset .$

   Next, if $J(\mathcal{F})$ is infinite and has an isolated point $w_0,$ say, then by Theorem \ref{expansive}, there exists an open subset $U$ in $F(\mathcal{F})$ such that $\mathbb{C}\setminus \left(\cup_{f\in\mathcal{F}}f(U)\right)$ has at most one point. We claim that $f(U)\cap J(\mathcal{F})=\emptyset$ for any $f\in \mathcal{F}.$ For, suppose $f(U)\cap J(\mathcal{F})\neq\emptyset$ for some $f\in\mathcal{F}$. Then there is  $w\in f(U)\cap J(\mathcal{F})$ such that $w=f(z)$ for some $z\in U.$ Since $w\in J(\mathcal{F})$, $z\in f^{-1}(\{w\})\subseteq J(\mathcal{F})$, a contradiction. Thus it follows that $J\subseteq \mathbb{C}\setminus \left(\cup_{f\in\mathcal{F}}f(U)\right)$, showing that $J(\mathcal{F})$ is finite which is not the case. Hence $J(\mathcal{F})$ has no isolated points.
\end{proof}

\subsection{Connectedness of Julia like set}

Kisaka \cite{Kisaka} characterized the connectedness of the Julia set of a transcendental entire function, as a subset of $\mC.$ Here, we also characterize the connectedness of Julia like set $J(\mathcal{F})$ of a family $\mathcal{F}$ of holomorphic functions on a simply connected domain in $\mC.$\\ 

Let $D$ be a domain in $\mathbb{C}$. Let $D_0$ be a subset of $D$. We shall denote by $\partial D_0$, the set of boundary points of $D_0$ in $D$ and denote by $\overline{D}_0$, the set of adherent points of $D_0$ in $D$.

\begin{lemma} \label{t1} Let $D$ be a simply connected domain in $\mathbb{C}$. Let $D_1$ and $D_2$ be two disjoint open connected subsets of $D$ such that $\partial D_1 \subset \partial D_2$. Then $\partial D_1$ is connected.
\end{lemma}

\begin{proof} Suppose on the contrary that  $\partial D_1=A\cup B,$ where $A$ and $B$ be two nonempty disjoint closed subsets of $\partial D_1.$ Let $\zeta_1\in B$. Since $A$ is closed and $\left\{\zeta_1\right\}$ is compact, $d(A,\zeta_1)=\epsilon>0$ (where $d$ is the Euclidean metric) and so we can choose $z_1\in D_1$ and $z_2\in D_2$ with $d(\zeta_1,z_i)<\frac{\epsilon}{2}\ (i=1,2)$ and a line segment $L_1$ joining $z_1$ and $z_2$. Clearly $L_1\cap A=\phi$. Similarly, we can choose  $\zeta_2\in A$, $z_{1}^{'}\in D_1 , z_{2}^{'}\in D_2$ and a line segment $L_2$ joining $z_{1}^{'} $ and $z_{2}^{'}$ with $L_2\cap B=\phi$. Since $z_1$ and $z_{1}^{'}$ are in $D_1$, there is a curve $\gamma_1\subset D_1$ joining $z_1$ and $z_{1}^{'}$ such that $\gamma_1$ does not intersect $L_1$ and $L_2$. Similarly, we can choose a curve $\gamma_2\subset D_2$ joining $z_2$ and $z_{2}^{'}$ such that $\gamma_2$ does not intersect $L_1$ and $L_2$. 
 Let $U$ be the region bounded by the closed curve $\gamma_1 \cup \gamma_2\cup L_1\cup L_2$. Then $A\cap \overline{U}$ and $B\cap \overline{U}$ are compact and hence are at a positive distance apart. Now from this it follows that  we can choose a curve $\gamma_3\subset D$ joining a point at $\gamma_1$ and a point at $\gamma_2$ and which does not intersect $\partial D_1=A\cup B$ which implies that $D_1\cap D_2\neq \phi,$ a contradiction.
\end{proof}

Simple connectedness of $D$ in Theorem \ref{t1} is essential:
\begin{example} Consider the annulus $D=\left\{z\in \mC:r_1<|z|<r_2\right\}$, where $0<r_1<r_2$ and consider $D_1=D\cap \left\{z:Im(z)>0\right\}$ and $D_2=D\cap \left\{z:Im(z)<0\right\}$ as two disjoint open connected subsets of $D$. Then $\partial D_1\subset \partial D_2$, and $\partial D_1$ is not connected.
\end{example}

As an immediate consequence of Lemma \ref{t1}, we have
\begin{lemma} \label{t2} Let $D$ be a simply connected domain in $\mathbb{C}$ and $D_1$ be an open connected subset of $D$. If $U$ is an component of  $D\setminus \overline{D}_{1},$ then $\partial U$ is connected.
\end{lemma}


\begin{theorem} \label{t3} Let $K$ be a closed subset of a simply connected domain $D$ in $\mC.$ Then $K$ is connected if, and only if the boundary of each component of the complement  $D\setminus K$ of $K$ is connected.
\end{theorem}

\begin{proof} The connectedness of $K$ is achieved, without any significant modification, by following the proof of Proposition $1$ in \cite{Kisaka}. The converse is proved by using the ideas of Newman(\cite{New}, Theorem 14.4) as  follows:\\ 
Suppose on the contrary that there is a component $G$ of $D\setminus K$ with disconnected boundary. Let $A$ be the component of $\partial G$ and put $B:=\partial G \setminus A$. Since $\partial (D\setminus \overline G)=\partial\overline G\subset \partial G$, Lemma \ref{t2} implies that the boundary of any component of $D\setminus \overline G$ does not meet $A$ and $B$ simultaneously which leads to a natural division of the class $\mathcal{C}$ of components of  $D\setminus \overline G$ into two subclasses:
$$\mathcal{C}_1:=\left\{U\in \mathcal{C}: \partial U\subset A \right\},$$
and 
$$\mathcal{C}_2:=\left\{V\in \mathcal{C}: \partial V\subset B \right\}.$$
 Put
$$U_1:=\bigcup_{U\in \mathcal{C}_1} U,$$ and 
$$U_2:=\bigcup_{V\in \mathcal{C}_2} V.$$

 {\bf Claim:} $ U_1\cup A$ and $U_2\cup B$ are closed sets.\\
 First, we show that $\partial U_1\subset A$. For, let $z_0\in \partial U_1$. We consider the following two cases:\\ 
\textbf{Case-I:} There exists a component $D_0\in \mathcal{C}_1$ such that $z_0\in\partial D_0$ and hence $\partial U_1\subset A$.\\
\textbf{Cases-II:} There does not exist a component $D_0\in \mathcal{C}_1$ such that $z_0\in\partial D_0.$ Then for each neighborhood $N_0:=\left\{z:|z-z_0|<\epsilon\right\}$ of $z_0$, we see that $N_0\cap U_1\neq\phi$. Thus, there exists a component $D_1\in \mathcal{C}_1$ such that $N_0\cap D_1\neq\phi$. This implies that for each $n\in\mN$ there is a component $D_n \in \mathcal{C}_1$ with $z_0 \notin \partial D_n$ such that $N_n\cap D_n\neq\phi$, where  $N_n=\left\{|z-z_1|<\epsilon/n\right\}.$  Hence $N_n\cap \partial D_n\neq \phi$. Since $\partial D_n\subset A$, $N_n\cap A\neq\phi$. That is, $N_n\cap A\neq\phi, \forall \ n\in \mN$. This implies that $z_0\in\overline{A}=A.$ Thus  $\partial U_1\subset A$, as desired.\\

Thus, $U_1\cup \partial U_1\cup A=U_1\cup A$ is closed. Similarly, $U_2\cup B$ is closed, and hence the claim.\\
Further, we have
$$\left(U_1\cup A\right)\cup\left(U_2\cup B\right)=\left(D\setminus\overline D_1\right)\cup \partial D_1=\left(D\setminus D_1\right)^{o}\cup \partial \left(D\setminus D_1\right)=D\setminus D_1.$$
Since $D\setminus D_1$ contains $K$, the union $\left(U_1\cup A\right)\cup\left(U_2\cup B\right)$ contains $K$. Since $A$ and $B$ are non empty disjoint subsets of $K$, $\left(U_1\cup A\right)\cap K$ and $\left(U_2\cup B\right)\cap K$ are non empty disjoint closed subsets of $K$ whose union is equal to $K$ showing that $K$ is disconnected, a contradiction. 
\end{proof}
 From Theorem \ref{t3}, we immediately obtain the connectedness of $J(\mathcal{F})$ as follows:
\begin{theorem} Let $\mathcal{F}$ be a family of holomorphic functions in a simply connected domain $D$. Then $J(\mathcal{F})$ is connected if, and only if the boundary of each component of $F(\mathcal{F})$ is connected.
\end{theorem}
\section{Escaping like set, generalized escaping like set and their properties}\label{sec3}

In the following discussion, by an infinite sequence in a subfamily $\mathcal{F}\subseteq \mathcal{H}(D)$ we mean a sequence $\{f_n\}\subset \mathcal{F}$ with $f_m\neq f_n \mbox{ for } m\neq n.$
\begin{definition}\label{def:escaping}
For a subfamily $\mathcal{F}\subseteq \mathcal{H}(D)$, we define  {\it escaping like set} and  {\it generalized escaping like set} of $\mathcal{F}$ as:
 $$I\left(\mathcal{F}\right):=\left\{z\in D:f_n(z)\to\infty  \mbox{ for  every infinite sequence } \left\{f_n\right\} \mbox{ in }\mathcal{F}\right\},$$ 
and 
$$U\left(\mathcal{F}\right):=\left\{z\in D:f_n(z)\to\infty \mbox{ for  some  sequence } \left\{f_n\right\} \mbox{ in }\mathcal{F}\right\},$$
respectively. \\
\end{definition}

\begin{remark}
$(i)$ ~  $I(\mathcal F)\subset U(\mathcal F).$\\
$(ii)$ ~  If $\mathcal{F}_1$ and $\mathcal{F}_2$ are two subfamilies of $\mathcal{H}(D),$ then the following hold:
\begin{enumerate}
  \item If $\mathcal{F}_1\subseteq \mathcal{F}_2$, then $I\left(\mathcal{F}_2\right)\subset I\left(\mathcal{F}_1\right)$ and $U\left(\mathcal{F}_1\right)\subset U\left(\mathcal{F}_2\right)$. 
	\item $I\left(\mathcal{F}_1\cup\mathcal{F}_2\right)=I\left(\mathcal{F}_1\right)\cap I\left(\mathcal{F}_2\right)$ and $U\left(\mathcal{F}_1\cup\mathcal{F}_2\right)=U\left(\mathcal{F}_1\right)\cup U\left(\mathcal{F}_2\right)$.
	\item If $\mathcal{F}_1\cap\mathcal{F}_2$ is infinite, then $I\left(\mathcal{F}_1\cap \mathcal{F}_2\right)\supseteq I\left(\mathcal{F}_1\right)\cup I\left(\mathcal{F}_2\right).$
		\item $U\left(\mathcal{F}_1\cap\mathcal{F}_2\right)\subset U\left(\mathcal{F}_1\right)\cap U\left(\mathcal{F}_2\right)$.
\end{enumerate}
\label{remark(ii)}
\end{remark}
Following examples show that the equality need not  hold in $(3)$ and $(4)$ in Remark \ref{remark(ii)}:
\begin{example} Consider
 $$\mathcal{F}_1:=\left\{e^{nz}:n\in\mathbb{N}\right\}\cup\left\{z^n:n\in\mathbb{N}\right\}$$ 
and 
$$\mathcal{F}_2:=\left\{e^{nz}:n\in\mathbb{N}\right\}\cup\left\{(z-x)^n:n\in\mathbb{N}\right\},$$
 where $x>1$ is chosen such that the disk $\left\{z:|z-x|<1\right\}$ intersects the unit disk $\left\{z:|z|<1\right\}.$ Then 
$$I\left(\mathcal{F}_1\right)=\left\{z:Re(z)>0\mbox{ and }|z|>1\right\},$$ 
$$I\left(\mathcal{F}_2\right)=\left\{z:Re(z)>0\mbox{ and }|z-x|>1\right\}$$ and 
$$I\left(\mathcal{F}_1\cap\mathcal{F}_2\right)=\left\{z:Re(z)>0\right\}.$$ Clearly, $I\left(\mathcal{F}_1\cap \mathcal{F}_2\right)\neq I\left(\mathcal{F}_1\right)\cup I\left(\mathcal{F}_2\right).$
\end{example}
\begin{example}
Consider the subfamilies 
$$\mathcal{F}_1:=\left\{nz:n\in\mathbb{N}\right\}$$
and 
$$\mathcal{F}_2:=\left\{n(z-\frac{1}{2}):n\in\mathbb{N}\right\}$$
of $\mathcal{H}(\mD).$ Then $U\left(\mathcal{F}_1\cap \mathcal{F}_2\right)=\phi$ and $U\left(\mathcal{F}_1\right)\cap U\left(\mathcal{F}_2\right)=\left\{z:|z|<1\right\}\setminus \left\{0,\frac{1}{2}\right\}$. Therefore, $U\left(\mathcal{F}_1\cap\mathcal{F}_2\right)\neq U\left(\mathcal{F}_1\right)\cap U\left(\mathcal{F}_2\right)$.
\end{example}

 Further, $I(\mathcal{F})$ and $F(\mathcal{F})$  possess the following-easy to verify-properties:
\begin{itemize}
\item[(a)] $F\left(\mathcal{F}_1+\mathcal{F}_2\right)=F(\mathcal{F}_1)\cap F(\mathcal{F}_2).$
\item[(b)] $I\left(\mathcal{F}_1+\mathcal{F}_2\right)\subseteq I(\mathcal{F}_1)\cap I(\mathcal{F}_2).$
\item[(c)] $F\left(\mathcal{F}_1 \mathcal{F}_2\right)=F(\mathcal{F}_1)\cup F(\mathcal{F}_2).$
\item[(d)] $I\left(\mathcal{F}_1\mathcal{F}_2\right)\subseteq I(\mathcal{F}_1)\cup I(\mathcal{F}_2).$
\end{itemize}


	
	
  
If $z\in I(\mathcal{F})\cap F(\mathcal{F})$, then by the definition of $I(\mathcal{F})$ and the normality of $\mathcal{F}$ at $z$ imply that the component of $F(\mathcal{F})$ which contains $z$ is contained in $I(\mathcal{F}).$ This conclusion also holds for $U(\mathcal{F}).$ That is,
\begin{itemize}
\item[1.] If $I(\mathcal{F})\cap F(\mathcal{F})\neq\emptyset$, then $I(\mathcal{F})$ has non-empty interior. Moreover, if  $U\cap I(\mathcal{F})\neq\emptyset$ for some component $U$ of $F(\mathcal{F})$, then $U\subseteq I(\mathcal{F}).$\\
\item[2.] If $U(\mathcal{F})\cap F(\mathcal{F})\neq\emptyset$, then $U(\mathcal{F})$ has non-empty interior. Moreover, if  $V\cap U(\mathcal{F})\neq\emptyset$ for some component $V$ of $F(\mathcal{F})$, then $V\subseteq U(\mathcal{F}).$
\end{itemize}
As a consequence of the above conclusions, one can see that if $J(\mathcal{F})=\emptyset,$ then $I(\mathcal{F})$ and $U(\mathcal{F})$ are  open subsets of $D.$

\smallskip

In general, $I(\mathcal{F})$ is neither  forward invariant nor backward invariant, for example, consider the family
 $\mathcal{F}:=\left\{e^{nz}:n\in\mathbb{N}\right\}.$ Then 
$$I(\mathcal{F})=\left\{z\in\mathbb{C}:Re(z)>0\right\}.$$ 
Since exponential function maps vertical lines onto circles, $I(\mathcal{F})$ is not forward invariant. Again, since exponential function maps horizontal lines onto rays emanating from the origin, $I(\mathcal{F})$ is not backward invariant. However, we have

\begin{theorem}\label{invariance} If  $\mathcal{F}$ is a family of entire functions such that $f\circ g=g\circ f$, for each $f,g\in\mathcal{F}$, then $I(\mathcal{F})$ and $U(\mathcal{F})$ are backward invariant. 
\end{theorem}
\begin{proof} Let $w\in I(\mathcal{F})$ and $g\in\mathcal{F}$. Let $z\in g^{-1}(\{w\})$ be such that $z\notin I(\mathcal{F})$. Then there exists a sequence $\left\{f_n\right\}$ in $\mathcal{F}$  which is bounded at $z$. Since $g$ is continuous, $\left\{g\circ f_n\right\}$ is bounded at $z$. But $g\circ f_n=f_n\circ g$, so the sequence ${f_n}$ is bounded at $g(z)=w$, a contradiction. This proves that $I(\mathcal{F})$ is backward invariant.

Let $w\in U(\mathcal{F})$ and $g\in\mathcal{F}$. Let $z\in g^{-1}(\{w\})$ be such that $z\notin U(\mathcal{F})$. Then each sequence $\left\{f_n\right\}$ in $\mathcal{F}$ is bounded at $z$. By the same argument as above, we find that $U(\mathcal{F})$ is backward invariant.
\end{proof}

For semigroups of transcendental entire functions, we have
\begin{theorem} \label{trans} Let $\mathcal{F}$ be a semigroup of transcendental entire functions. Then $U(\mathcal{F})$ is non-empty and backward invariant. Further, if $\mathcal{F}=\left\langle f_1,\cdots,f_m\right\rangle$, where $f_i$ are transcendental entire functions, then for each $z\in U(\mathcal{F}),$ there exists  $f_i\in\left\{f_1,\cdots, f_m\right\}$ such that $f_i(z) \in U(\mathcal{F}).$ 
\end{theorem}
\begin{proof} Let $f\in\mathcal{F}$. Then $I(f)\neq\phi$ by Theorem 1 of Eremenko\cite{eremenko} and hence $U(\mathcal{F})\neq\phi.$ Let $z_1\in{U}(\mathcal{F})$ and $f\in\mathcal{F}$. Put $w_1\in f^{-1}(\{z_1\})$. Then there is a sequence $\left\{f_n\right\}$ in $\mathcal{F}$ such that $f_n(z_1)\to\infty$ as $n\to\infty$. Put $g_n=f_n\circ f.$ Then $g_n\in\mathcal{F},\ \forall\ n\in \mathbb{N}$. Further, $g_n(w_1)=f_n(z_1)\to\infty,$ as $ n\to\infty$ showing that $w_1\in{U}(\mathcal{F})$ and hence ${U}(\mathcal{F})$ is backward invariant. 

Further, let $z_0\in{U}(\mathcal{F})$. Then there is a sequence $\left\{g_n\right\}\subset\mathcal{F}$ such that $g_n(z_0)\to\infty,$ as $n\to\infty$ and hence there exists an $n_0\in\mathbb{N}$ such that $g_n\neq f_i, \forall \ i=1,\ldots, m, \forall \ n\geq n_0$. This implies that for each $n\geq n_0$, 
$$g_n=h_n\circ f_{i},\mbox{ for some } i\in\left\{1,\ldots, m\right\}, \mbox{ and for some } h_n\in\mathcal{F}.$$
Then we can choose a subsequence $\left\{g_{n_{k}}\right\}$ of $\{g_n\}$ such that $g_{n_k}=h_{n_k}\circ f_{i_0}$, for some fixed $i_0\in\left\{1,\ldots, m\right\}$. Let $w_0=f_{i_0}(z_0)$. Then $h_{n_k}(w_0)=h_{n_k}\circ f_{i_0}(z_0)=g_{n_k}(z_0)\to\infty,$ as $k\to\infty$ showing that $w_0\in{U}(\mathcal{F})$. Hence $f_i(z_0)\in U(\mathcal{F})$, for some $f_i\in\left\{f_1,\ldots,f_m\right\}$.
\end{proof}

\medskip
If  $I(\mathcal{F})$ and  $U(\mathcal{F})$ are not open subsets of $D$, then one can easily see that  $I(\mathcal{F})$ as well as  $U(\mathcal{F})$ intersect $J(\mathcal{F}).$ Converse of this statement does not hold as seen through the following examples:
\begin{example} \label{ex1}$(i)$ Let 
$$U=\left\{z:|z-2|<1/2\right\},$$
$$f_n(z):=\left[\left(2-\frac{1}{n}\right)-z\right]^2z^n, \ z\in U,$$
and consider the family
$$\mathcal{F}:=\{f_n:n\in \mN\}.$$ 

Then $f_n(z)\to\infty,$ as  $n\to\infty,\ z\in U.$ But $\{f_n(z)\}$ does not tend to infinity uniformly in any neighborhood of $2$.  Thus $2\in J(\mathcal{F})$ and  $I(\mathcal{F})=U$. Also, note that $I(\mathcal{F})$ is open.

\smallskip

$(ii)$ Consider 
$$\mathcal{F}:=\left\{nz: n\in\mathbb{N}\right\}\cup\left\{n(z-1): n\in\mathbb{N}\right\}.$$ Then $J(\mathcal{F})=\left\{0,1\right\}$ and $U(\mathcal{F})=\mathbb{C}$. Thus $U(\mathcal{F})\cap J(\mathcal{F})\neq\phi$ and $U(\mathcal{F})$ is open.
\end{example}

Following example shows that  $I(\mathcal{F})$  may be empty or non-empty independent of whether $J(\mathcal{F})$ is empty or non-empty:
\begin{example} $(i)$ If $\mathcal{F}$ is locally uniformly bounded family of holomorphic functions on a domain $D$, then $J(\mathcal{F})=\emptyset$ and $I(\mathcal{F})=\emptyset$.

\smallskip

 $(ii)$ Consider the subfamily 
$$\mathcal{F}:=\left\{n(z-2):n\in\mathbb{N}\right\}$$
of $\mathcal{H}(\mathbb{D})$. Then $J(\mathcal{F})=\emptyset$ and $I(\mathcal{F})=\mathbb{D}$, which is non-empty.

\smallskip
 
$(iii)$ Consider the subfamily 
$$\mathcal{F}:=\left\{nz:n\in\mathbb{N}\right\}\cup\left\{n(z-1):n\in\mathbb{N}\right\}$$
of $\mathcal{H}(\mathbb{D})$. Then $J(\mathcal{F})=\left\{0,1\right\}$ and $I(\mathcal{F})=\mathbb{D}\setminus \left\{0,1\right\}.$ Thus both $J(\mathcal{F})$ and $I(\mathcal{F})$ are non-empty.

\smallskip

$(iv)$ For the family $\mathcal{F}:=\cup_{|a|<1}\left\{n(z-a):n\in\mathbb{N}\right\}$ in $\mathcal{H}(\mathbb{D}),$ we see that $J(\mathcal{F})=\mathbb{D}$ and $I(\mathcal{F})=\emptyset$. 
\end{example}

If $z\in\partial{I(\mathcal{F})}$, then clearly $\mathcal{F}$ is not normal at $z$ and hence $\partial{I(\mathcal{F})}\subseteq J(\mathcal{F})$. The other way inclusion may not hold, see $(i)$ of Example \ref{ex1}.

\medskip

\begin{theorem}\label{semigroup} Suppose that $\mathcal{F}$ is a semigroup of entire functions and $I(\mathcal{F})$ has at least two points and is invariant. Then $J(\mathcal{F})=\partial{I(\mathcal{F})} .$
\end{theorem} 
\begin{proof} Let $z\in \mathbb{C}\setminus I(\mathcal{F})$ and  let $f\in \mathcal{F}$ be such that $f(z)\in I(\mathcal{F})$. Then backward invariance of $I(\mathcal{F})$ implies that $z\in I(\mathcal{F})$, a contradiction. This implies that $\mathcal{F}$ omits $I(\mathcal{F})$ on $\mathbb{C}\setminus I(\mathcal{F})$. By Montel's theorem, open subsets of $\mathbb{C}\setminus I(\mathcal{F})$ are contained in $F(\mathcal{F})$. 
 
Since transcendental entire function has infinitely many periodic points,  $\mathbb{C}\setminus I(\mathcal{F})$ has at least two points. Forward invariance of $I(\mathcal{F})$ implies that $\mathcal{F}$ omits $\mathbb{C}\setminus I(\mathcal{F})$ on $I(\mathcal{F})$.
and hence by Motel's theorem,  open subsets of $I(\mathcal{F})$ are contained in $F(\mathcal{F})$. This implies that $J(\mathcal{F})\subseteq \partial I(\mathcal{F})$.  
\end{proof}
\begin{question} Under the hypothesis of Theorem \ref{semigroup}, can $J(\mathcal{F})$ be empty? In the dynamics of entire functions, it is always non-empty.
\label{q2}
\end{question}
 When $J(\mathcal{F})\neq\emptyset,$  the following result holds:

\begin{theorem}\label{nonemptyGES} If $\mathcal{F}$ is a subfamily of $\mathcal{H}(D)$ such that $J(\mathcal{F})$ has an isolated point, then $U(\mathcal{F})\neq\emptyset$.
\end{theorem}
\begin{proof} Suppose that $z_0$ is a isolated point of $J(\mathcal{F})$ and let $N$ be a neighborhood of $z_0$ such that $N\cap J(\mathcal{F})\setminus\left\{z_0\right\}=\phi.$  Let  $\left\{f_n\right\}$ be a sequence in $\mathcal{F}$ such that it has no uniformly convergent subsequence in $N$. 

\smallskip

We shall show that $f_n(z)\rightarrow\infty$ in $N\setminus \left\{z_0\right\}$. Suppose on the contrary that there is a subsequence $\left\{f_{n_k}\right\}$ of $\left\{f_{n}\right\}$ and a point $z_1\in N\setminus\left\{z_0\right\}$ such that $|f_{n_k}(z_1)|\leq M$ for all $k\in\mathbb{N}$ and for some $M>0$. By [\cite{chuang}, Lemma 2.9], we see that  $\left\{f_{n_k}\right\}$ is locally uniformly bounded in $N\setminus\left\{z_0\right\}$. Take a circle $C$ with center $z_0$ and radius $\epsilon$ in $N\setminus\left\{z_0\right\}$, there exists a constant $M_1>0$ such that $|f_{n_k}(z)|\leq M_1$ for all $z\in C$ and $n\in\mathbb{N}$. Then by Maximum Modulus Principle, $|f_{n_k}(z)|\leq M_1$ for all $z\in \left\{z:|z-z_0|<\epsilon\right\}$ and for all $n\in\mathbb{N}$. Thus $\left\{f_{n_k}\right\}$ is normal at $z_0$, a contradiction. 
\end{proof}
\begin{example}
Let $\mathcal{F}:=\left\{f_n(z)=nz:  n\in \mathbb{N}\right\}.$ Then $J(\mathcal{F})=\left\{0\right\}$ and $f_n(z)\rightarrow\infty, \ n\to\infty$, in any deleted neighborhood of $0$.
\end{example}
 
If $J(\mathcal{F})$ has an isolated point, it is implicit in the proof of Theorem \ref{nonemptyGES} that  $U(\mathcal{F})$ has non-empty interior. Consequently, we have:
\begin{corollary}\label{coro} If $U(\mathcal{F})$ has empty interior, then $J(\mathcal{F})$ is either empty or a perfect set.
\end{corollary}

\section{Discussion on limit functions and fixed points of $\mathcal{F}$}

 Let $\mathcal{F}$ be a subfamily of  $\mathcal{H}(D)$ and let $U$ be a component of $F(\mathcal{F}).$ A holomorphic function $f$ on $D$ is said to be a limit function of $\mathcal{F}$ on $U$ if there is a sequence $\left\{f_n\right\}$ in $\mathcal{F}$ which converges locally uniformly on $U$ to $f.$ If there is a sequence in $\mathcal{F}$ which converges locally uniformly to $\infty,$ then $\infty$ also qualifies to be a limit function of $\mathcal{F}.$
	By $\mathcal{L}_{\mathcal{F}}(U)$, we denote the set of finite limit functions of $\mathcal{F}$ on $U.$

Suppose that $f\circ g=g\circ f$ for every $f,g\in\mathcal{F}$ and $U$ is a forward invariant component of $F(\mathcal{F})$. If a constant $c$ is  a limit function of $\mathcal{F}$ on $U$, then one can see that either $c=\infty$ or $c$ is a fixed point of every $f\in\mathcal{F}$. Further, if $\mathcal{L}_{\mathcal{F}}(U)$ contains only constant functions, then $\mathcal{L}_{\mathcal{F}}(U)$ is a singleton.

\smallskip

A point $z_0\in D$ is said to be a fixed point of a subfamily $\mathcal{F}$ of $\mathcal{H}(D)$ if $z_0$ is a fixed point of each $f\in \mathcal{F}$. Classification of fixed points of an entire function can be extended to the fixed points of a family of holomorphic functions.  In classical dynamics, if $z_0$ is an attracting or repelling fixed point  of $f$, then $z_0$ is in Fatou set $F(f)$ or Julia set $J(f)$ of $f$ respectively. This is not true in this situation, even  a super attracting fixed point may not be in the Fatou like set $F(\mathcal{F}).$ For example, 
\begin{itemize}
\item [(i)] $0$ is an attracting (not super attracting) fixed point of
$$\mathcal{F}:=\left\{f_n(z)=\left(\frac{1}{2}+\frac{1}{3n}\right)ze^{nz}: n\in\mathbb{N}\right\}$$
and $0\in J(\mathcal{F})$; 
\item[(ii)] $0$ is a super attracting fixed point of 
$$\mathcal{F}:=\left\{f_n(z)=nz^{2}:n\in \mathbb{N}\right\}$$ and $0\in J(\mathcal{F})$; 
\item[(iii)]  $0$ is a repelling fixed point of 
$$\mathcal{F}:=\left\{nz:n\geq 2\right\}$$ and $0\in J(\mathcal F).$ 
\end{itemize}
If $z_0$ is a super attracting fixed point of a family $\mathcal{F}$ of holomorphic functions on a domain $D$ and $g$ is a non constant limit function of $\mathcal{F}$, then clearly $z_0$ is super attracting fixed point of $g$. But the same is not true if $z_0$ is an attracting fixed point of $\mathcal{F},$ for example $0$ is the attracting fixed point of 
$$\mathcal{F}:=\left\{f_n(z)=ze^{z}\left(1-\frac{1}{2n}\right): n\in\mathbb{N}\right\}$$ but $0$ is not  an attracting fixed point of the limit function $g(z)=ze^{z}$ of $\mathcal{F}.$ 
With regard to repelling fixed points, the Fatou like set may contain the repelling fixed points of $\mathcal{F},$    for example $0$ is a repelling fixed point of 
$$\mathcal{F}:=\left\{f_n(z)=a\left(1+\frac{1}{n}\right)ze^{z}: n\in\mathbb{N}\right\}, \  |a|> 1.$$




It is well known that a Fatou component contains at most one fixed point. But this is not true in Fatou like sets. That  is, a component $U$ of $F(\mathcal{F})$ can contain two fixed points, for example $0$ is an indifferent fixed point and $\frac{1}{2}$ is a repelling fixed point of
$$\mathcal{F}:=\left\{f_n(z)=z^{n}(z-\frac{1}{2})+z : n\in \mathbb{N}\right\}$$ and both lie in $F(\mathcal{F})$ since $J(\mathcal{F})=\{z:|z|=1\}.$ The Fatou like set may contain two attracting fixed points of $\mathcal{F},$ for example consider $g(z)=a+(z-a)h(z),$ where
$$h(z)=\frac{(a-b)(z-b)+(z-a)}{(b-a)}, \ a,b\in\mathbb{R}: 0<b-a<\frac{1}{2}.$$
Then $a$ and $b$ are attracting fixed points of $g$. Let $a=0.1$ and $b=-0.1,$ and let $f_n(z)=g(z)+\left((z-a)(z-b)\right)^{n},\ \forall \ n\in\mathbb{N}.$ Then $a,\ b$ are attracting fixed points of $\left\{f_n\right\}.$ Moreover, $\{f_n(z)\}$ converges uniformly to $g(z)$ in $\left\{z:|z|<0.3\right\}$. Let $U$ be the component of $F(\left\{f_n\right\})$ containing $\left\{z:|z|<0.3\right\}$. Then $U$ contains two attracting fixed points $a$ and $b$ of $\left\{f_n\right\}.$ 
\bibliographystyle{amsplain}

\end{document}